\documentclass[12pt,notitlepage]{amsart}
\usepackage{latexsym,amsfonts,amssymb,amsmath,amsthm} 
\newcommand{\mz}{\ensuremath{\mathbb Z}}
\newcommand{\mr}{\ensuremath{\mathbb R}}

\newcommand{\mc}{\ensuremath{\mathbb C}}

\newcommand{\shortmod}{\ensuremath{\negthickspace \negthickspace \negthickspace \pmod}}

\newcommand{\half}{\ensuremath{ \frac{1}{2}}}

\newcommand{\intR}{\int_{-\infty}^{\infty}}

\newcommand{\sumstar}{\sideset{}{^*}\sum}
\newcommand{\leg}[2]{\left(\frac{#1}{#2}\right)}
\newcommand{\e}[2]{e\left(\frac{#1}{#2}\right)}

\theoremstyle{plain}		
	\newtheorem{mytheo}{Theorem}[section]
	
	\newtheorem{myprop}[mytheo]{Proposition}
	
     \newtheorem{mylemma}[mytheo]{Lemma}

\theoremstyle{remark}

\numberwithin{equation}{section}
\begin{document}
\title{The second moment of $GL(3) \times GL(2)$ $L$-functions at special points}
\author{Matthew P. Young} 
\address{Department of Mathematics \\
	  Texas A\&M University \\
	  College Station \\
	  TX 77843-3368 \\
		U.S.A.}
\email{myoung@math.tamu.edu}
\thanks{Work supported by NSF grant DMS-0758235}

\begin{abstract}
For a fixed $SL(3,\mathbb{Z})$ Maass form $\phi$, we consider the family of $L$-functions $L(\phi \times u_j, s)$ where $u_j$ runs over the family of Hecke-Maass cusp forms on $SL(2,\mathbb{Z})$.  We obtain an estimate for the second moment of this family of $L$-functions at the special points $\half + it_j$ consistent with the Lindel\"{o}f Hypothesis.  We also obtain a similar upper bound on the sixth moment of the family of Hecke-Maass cusp forms at these special points; this is apparently the first occurrence of a Lindel\"{o}f-consistent estimate for a sixth power moment of a family of $GL(2)$ $L$-functions.
\end{abstract}

\maketitle

\section{Introduction}
In this paper we study the Rankin-Selberg convolution of a fixed Maass form on $SL(3, \mz)$ with the family of Hecke-Maass cusp forms on $SL(2,\mz)$.  At the subconvexity workshop held in October 2006 at the American Institute of Mathematics, A. Venkatesh posed the question of studying the second moment of this family at the special points $\half + it_j$ (where $\frac14 + t_j^2$ is the Laplace eigenvalue), as well as an integrated version which we treat in a companion paper \cite{Y}.  These points are of particular arithmetical interest because they are zeros of the Selberg zeta function (and so their behavior controls the error term in the prime geodesic theorem, c.f. \cite{IwPGT}) and from the Phillips-Sarnak theory of deformation of cusp forms \cite{PS}.  In the context of the subconvexity problem for $L$-functions, these points are special because the conductor ``drops'' (becomes relatively small).  This feature makes the subconvexity problem more difficult because the moment method requires taking a higher moment than for other points.  On the other hand, these special points potentially allow one to study higher moments than at typical points because the relevant arithmetical sums become shorter; this is the perspective taken in this work.

This family of Rankin-Selberg $L$-function has recently been studied by X. Li \cite{Li} \cite{Li2} to show subconvexity bounds for a self-dual $GL(3)$ $L$-function in $t$-aspect, as well as for the Rankin-Selberg $L$-function itself at the central point $s=1/2$.  Her method is to study the first moment of this family at the central point, heavily exploiting the nonnegativity of the central values.  Naturally one desires to estimate the second moment since one is no longer restricted to the central point to have nonnegativity.

One natural way to study the second moment of a family of $L$-functions is through an appropriate large sieve inequality, which itself presumably has more general applications.  For this family there are some partial results due to \cite{DI} and \cite{Luo}, which we improve further in this paper.  Actually, we noticed a curious similarity between this problem and certain aspects of the large sieve inequality for $\Gamma_1(q)$ obtained in \cite{IL} (which contains an auxiliary main term sometimes violating the rule of thumb of ``number of harmonics plus length of sum'' for large sieve inequalities).  It would be of great interest to better-understand these main terms.  Evidently the large sieve inequalities on $GL(2)$ are more delicate than on $GL(1)$.

\section{Notation}
We refer to \cite{Goldfeld} for the material and notation on $GL(3)$ Maass forms.  Suppose $\phi$ is a Maass form for $SL(3, \mz)$ of type $(\nu_1, \nu_2) \in \mc^{2}$,
which is an eigenfunction of all the Hecke operators.  The Godement-Jacquet $L$-function associated to $\phi$ is
\begin{equation}
 L(\phi, s) = \sum_{n=1}^{\infty} \frac{A(1,n)}{n^s} = \prod_{p} (1 - A(1,p)p^{-s} + A(p,1)p^{-2s} - p^{-3s})^{-1}.
\end{equation}
Here $A(m,n)$ are the Fourier coefficients normalized as in \cite{Goldfeld}.  In particular, $A(1,1) =1$ and $A(m,n)$ are constant on average (see Remark 12.1.8 of \cite{Goldfeld}).  The dual Maass form $\widetilde{\phi}$ is of type $(\nu_2, \nu_1)$ and has $A(n,m) = \overline{A(m,n)}$ as its $(m,n)$-th Fourier coefficient, whence
$ L(\widetilde{\phi}, s) = \sum_{n} A(n,1) n^{-s}$. 
Letting
\begin{equation}
 \Gamma_{\nu_1, \nu_2}(s) = \pi^{-3s/2} \Gamma\leg{s + 1 - 2 \nu_1 - \nu_2}{2} \Gamma\leg{s + \nu_1 - \nu_2}{2} \Gamma\leg{s - 1 + \nu_1 + 2\nu_2}{2},
\end{equation}
the functional equation for $L(\phi, s)$ reads
\begin{equation}
 \Gamma_{\nu_1, \nu_2}(s) L(\phi, s) = \Gamma_{\nu_2, \nu_1}(1-s) L(\widetilde{\phi}, 1-s).
\end{equation}

Let $(u_j)$ be an orthonormal basis of Hecke-Maass cusp forms on $SL(2,\mz)$ with corresponding Laplace eigenvalues $\frac14 + t_j^2$.  Let $\lambda_j(n)$ be the Hecke eigenvalue of the $n$-th Hecke operator for the form $u_j$.  Since the Hecke operators on $GL(2)$ are self-adjoint, the $\lambda_j(n)$'s are real.  Then
 $L(u_j,s) = \sum_{n} \lambda_j(n) n^{-s}$
satisfies a functional equation relating to $L(u_j,1-s)$.

As explained in Chapter 12.2 of \cite{Goldfeld}, the Rankin-Selberg convolution of $\phi$ and $u_j$ is
\begin{equation}
\label{eq:Rankin}
 L(u_j \times \phi, s) = \sum_{m,n =1}^{\infty} \frac{\lambda_j(n) A(m,n)}{(m^2 n)^s}.
\end{equation}
The completed $L$-function associated to $L(u_j \times \phi, s)$, for $u_j$ even, takes the form
\begin{multline}
 \Lambda(u_j \times \phi, s) = \pi^{-3s} \Gamma\leg{s - i t_j - \alpha}{2} \Gamma\leg{s - i t_j - \beta}{2} \Gamma\leg{s - i t_j - \gamma}{2}
\\
\Gamma\leg{s + i t_j - \alpha}{2} \Gamma\leg{s + i t_j - \beta}{2} \Gamma\leg{s + i t_j - \gamma}{2} L(u_j \times \phi, s),
\end{multline}
where $\alpha = -\nu_1 - 2 \nu_2 + 1$, $\beta = -\nu_1 + \nu_2$, and $\gamma = 2 \nu_1 + \nu_2 -1$ (see Theorem 12.3.6 of \cite{Goldfeld} for the explicit gamma factors).  Then this Rankin-Selberg convolution has a holomorphic continuation to $s \in \mc$ and satisfies the functional equation
\begin{equation}
 \Lambda(u_j \times \phi, s) = \Lambda(u_j \times \widetilde{\phi}, 1-s).
\end{equation}
The case with $u_j$ odd is similar, having slightly different constants $\alpha,\beta, \gamma$.


\section{Main result}
Our main result is
\begin{mytheo}
\label{thm:mainthm}
 We have
\begin{equation}
\label{eq:secondmoment}
 \sum_{ t_j \leq T} |L(u_j \times \phi, \tfrac12 + i t_j)|^2 \ll T^{2 + \varepsilon}.
\end{equation}
\end{mytheo}
The convexity bound for $L(u_j \times \phi, \half + i t_j)$ is $t_j^{3/4 + \varepsilon}$ due to a conductor-dropping phenomenon.  In practice this means that the arithmetical sums are somewhat shortened yet the coefficients remain to be genuine $GL(3) \times GL(2)$ objects, which allows us to study these fascinating arithmetical coefficients in a more analytically-tractible setting.

The Phillips-Sarnak \cite{PS} theory of deformation of cusp forms provides a motivation for studying the nonvanishing of the Rankin-Selberg convolution of a 
$GL(2)$ holomorphic cusp form with this family of Maass forms at the special point $\half + i t_j$, but I do not know of any application for the convolution with a $GL(3)$ form as in Theorem \ref{thm:mainthm}.

The method of proof of Theorem \ref{thm:mainthm} also shows
\begin{mytheo}
\label{thm:sixthmoment}
\begin{equation}
\label{eq:sixthmoment}
 \sum_{ t_j \leq T} |L(u_j, \tfrac12 + i t_j)|^6 \ll T^{2 + \varepsilon}.
\end{equation}
\end{mytheo}
This is apparently the first instance of an estimation for the sixth moment of a family of $GL(2)$ $L$-functions that is consistent with the Lindel\"{o}f hypothesis.

Deshouillers-Iwaniec \cite{DI} obtained a large sieve inequality that could obtain the estimate \eqref{eq:sixthmoment} but for the fourth moment.  Luo \cite{Luo} obtained an improved large sieve inequality which shows that the eight moment is bounded by $T^{5/2 + \varepsilon}$, and the sixth moment is bounded by $T^{9/4 + \varepsilon}$.  
In this work we further develop Luo's large sieve inequality at special points into a form conducive for further analysis using special properties of the vector under consideration.
For our application we then use the $GL(3)$ Voronoi summation formula of \cite{MS} to reveal extra cancellation.  It is worth mention that the estimate
\begin{equation}
 \sum_{T < t_j \leq T + 1} |L(u_j, \tfrac12 + i t_j)|^4 \ll T^{1 + \varepsilon}
\end{equation}
follows from a relevant short interval large sieve inequality stated without proof by Iwaniec \cite{I2} and proved independently by Jutila \cite{Jutila} and Luo \cite{Luo2}.  However, this inequality does not imply Theorem \ref{thm:sixthmoment} (though one does immediately recover the convexity bound by dropping all but one term).

We remark that our proof of Theorem \ref{thm:mainthm} uses summation formulas for each of $GL(1)$ (Poisson summation), $GL(2)$ (Kuznetsov formula), and $GL(3)$ (Voronoi summation).

The majority of this paper is used to prove Theorem \ref{thm:mainthm}, while in Section \ref{section:sixth} we briefly indicate the changes necessary to prove Theorem \ref{thm:sixthmoment}.

As a convention, throughout the paper $\varepsilon$ denotes an arbitrarily small positive constant that may vary from line to line.  

\subsection{Acknowledgements}
I would like to thank Brian Conrey, Xiaoqing Li, Wenzhi Luo, and Akshay Venkatesh for taking an interest in this work.

\section{Approximate functional equation}
In order to represent the special value we use an approximate functional equation.  Write $\lambda_{u_j \times \phi}(n)$ for the coefficient of $n^{-s}$ in the Dirichlet series \eqref{eq:Rankin} for $L(u_j \times \phi, s)$.  The standard approximate functional equation (Theorem 5.3 of \cite{IK}) says
\begin{equation}
\label{eq:AFE}
 L(u_j \times \phi, \tfrac12 + i t_j) = \sum_n \frac{\lambda_{u_j \times \phi}(n)}{n^{\half + i t_j}} V_{\half + i t_j}\leg{n}{Y} + \epsilon_j \sum_n \frac{\lambda_{u_j \times \widetilde{\phi}}(n)}{n^{\half - i t_j}} V_{\half - i t_j}^{*}(nY),
\end{equation}
for any $Y > 0$, where $V_{s}(y)$ and $V_{s}^{*}(y)$ are certain explicit smooth functions, and $\epsilon_j$ is a certain complex number of absolute value $1$.  
Precisely,
\begin{equation}
\label{eq:V}
 V_{\half + it_j}(y) = \frac{1}{2 \pi i} \int_{(3)} y^{-s} \frac{\gamma(\half + i t_j + s)}{\gamma(\half + i t_j)} \frac{G(s)}{s} ds,
\end{equation}
where $\Lambda(u_j \times \phi, s) = \gamma(s) L(u_j \times \phi, s)$ and $G(s)$ is an entire function with rapid decay in the imaginary direction.  Here $V^{*}_{\half - i t_j}$ has a similar form to $V_{\half + it_j}$ but with $\gamma(s)$ replaced by $\gamma^*(s)$, where $\Lambda(u_j \times \widetilde{\phi}, s) = \gamma^*(s) L(u_j \times \widetilde{\phi}, s)$.

An exercise in Stirling's approximation shows
\begin{equation}
\label{eq:Stirling}
 G(s) \frac{\gamma(\half + i t_j + s)}{\gamma(\half + i t_j)} = t_j^{\frac{3s}{2}} h(s)(1 + \frac{c_1(s)}{t_j} + \dots),
\end{equation}
where $h(s)$ is a holomorphic function with exponential decay as $\text{Im}(s) \rightarrow \infty$ in any fixed strip (with appropriate choice of $G$), and
each $c_i(s)$ is a polynomial in $s$.  This is valid for $\text{Re}(s) > 0$, say.  By inserting \eqref{eq:Stirling} into \eqref{eq:V} we obtain an asymptotic expansion for $V_{\half + it_j}(y)$ with leading term
\begin{equation}
\label{eq:V1}
 V_1\leg{y}{t_j^{3/2}} = \frac{1}{2 \pi i} \int_{(3)} \leg{t_j^{3/2}}{y}^{s}  \frac{h(s)}{s} ds.
\end{equation}
Of course $V^*$ has a similar expansion with leading order function $V_2$, say.

\section{Initial cleaning}
\label{section:cleaning}
In this section we make a series of simplifying reductions on the mean value \eqref{eq:secondmoment} using the approximate functional equation \eqref{eq:AFE} to represent the special values.  First, note that in proving Theorem \ref{thm:mainthm} it suffices to consider the dyadic segment $T < t_j \leq 2T$, and to treat $u_j$ even and odd separately.  For simplicity we treat the even case.  By taking a smooth partition of unity to the $n$-sums, and using Cauchy's inequality, it suffices to estimate 
\begin{equation}
\label{eq:cleaning1}
\sum_{\substack{T < t_j \leq 2T \\ u_j \text{ even}}} \left| \sum_{P < n \leq 2P}  \frac{\lambda_{u_j \times \phi}(n)}{n^{\half + i t_j}} w_1\leg{n}{P} V_{\half + i t_j}\leg{n}{Y} \right|^2 + \left| \sum_{P < n \leq 2P} \frac{\lambda_{u_j \times \widetilde{\phi}}(n)}{n^{\half - i t_j}} w_1\leg{n}{P} V^{*}_{\half - i t_j}(nY) \right|^2,
\end{equation}
where $P \ll T^{3/2 + \varepsilon}$, and the truncation $P < n \leq 2P$ is assumed to be redundant to the support of $w_1$.  Our next goal is to remove the dependence on $j$ in the weight functions as this will allow us to directly quote relevant results from the literature.  

Next we insert the asymptotic expansion for $V_{\half + i t_j}$ with leading term given by $V_1$ (and similarly for $V^*$).  For simplicity we treat only the leading-order term since all the other terms are of the same form but even smaller.  The contribution of this leading-order term to \eqref{eq:cleaning1} takes the form
\begin{equation}
\sum_{\substack{T < t_j \leq 2T \\ u_j \text{ even}}} \left| \sum_{P < n \leq 2P} \frac{\lambda_{u_j \times \phi}(n)}{n^{\half + i t_j}} w_1\leg{n}{P} V_1\leg{n}{Y t_j^{3/2}} \right|^2 + \left| \sum_{P < n \leq 2P}  \frac{\lambda_{u_j \times \widetilde{\phi}}(n)}{n^{\half - i t_j}} w_1\leg{n}{P} V_{2}\leg{nY}{t_j^{3/2}} \right|^2.
\end{equation}
We use a trick of \cite{DI} (see their Section 16): integrate with respect to $\frac{dY}{Y}$ from $Y_j=T^{3/2}/t_j^{3/2}$ to $eY_j$.  For the integral of the first term perform the change of variables $Y \rightarrow Y T^{3/2}/t_j^{3/2}$ (so now $1 \leq Y \leq e$) while for the latter do $Y \rightarrow Y t_j^{3/2}/T^{3/2}$ (so now $(T/t_j)^{3} \leq Y \leq e (T/t_j)^{3}$).  Then use positivity to extend both integrals to $\frac{1}{8} \leq Y \leq e$.  At this point $t_j$ has been eliminated from the weight function, which now takes the form $w_1(\tfrac{n}{P}) V_i(\tfrac{n y}{T^{3/2}})$ where $8^{-1} \leq y \leq 8$.  Letting $w_2(x) = w_1(x) V_i(x y \tfrac{P}{T^{3/2}})$, we see that $w_2$ satisfies the same essential properties as $w_1$, namely that it is smooth of compact support and each of its derivatives is bounded (independently of $P$).  Thus the problem reduces to estimating
\begin{equation}
 \sum_{\substack{T < t_j \leq 2T \\ u_j \text{ even}}} \left| \sum_{P < n \leq 2P}  \frac{\lambda_{u_j \times \phi}(n)}{n^{\half + it_j}}  w_2\leg{n}{P} \right|^2.
\end{equation}

It is also convenient to introduce the weight $\alpha_j = |\rho_j(1)|^2/\cosh(\pi t_j)$ which naturally appears in the Kuznetsov formula, where $\rho_j(1)$ is the first Fourier coefficient of the $u_j(z)$ (which recall has $L^2$ norm $1$).  These weights satisfy $t_j^{-\varepsilon} \ll \alpha_j \ll t_j^{\varepsilon}$ due to \cite{HL} and \cite{I}.  We also introduce the nonnegative smooth weight function
\begin{equation}
 w(t_j) = 2 \frac{\sinh((\pi - \frac{1}{T}) t_j)}{\sinh(2 \pi t_j)},
\end{equation}
satisfying $w(t_j) \sim \frac{\exp(-\pi t_j/T)}{\cosh(\pi t_j)}$ for $T < t_j \leq 2T$, and then remove this sharp truncation and the condition that $u_j$ is even, by positivity.  It therefore suffices to estimate
\begin{equation}
 H=\sum_{t_j} w(t_j) |\rho_j(1)|^2 \left| \sum_{P < n \leq 2P} w_2\leg{n}{P} \frac{\lambda_{u_j \times \phi}(n)}{n^{\half + it_j}}  \right|^2,
\end{equation}
the necessary bound being $H \ll T^{2 + \varepsilon}$.  Using \eqref{eq:Rankin} and Cauchy's inequality, we get
\begin{equation}
\label{eq:cleaning6}
H
\ll \log{P} \sum_{l \leq \sqrt{2P}} \frac{1}{l} \sum_{t_j} w(t_j) |\rho_j(1)|^2  \left| \sum_{P/l^2 < n \leq 2P/l^2} w_2\leg{ n}{P/l^2} \frac{A(l,n)\lambda_{j}(n)}{n^{\half + it_j}}  \right|^2.
\end{equation}
For each $l$, let, with $N = P/l^2$,
\begin{equation}
 H_l = \sum_{t_j} w(t_j) |\rho_j(1)|^2  \left| \sum_{N < n \leq 2N} w_2\leg{n}{N} \frac{A(l,n)\lambda_{j}(n)}{n^{\half + it_j}}  \right|^2.
\end{equation}
Now $H_l$ is in the form to which we can quote some results of Luo \cite{Luo}.  We will show
\begin{equation}
\label{eq:Sm}
 H_l \ll T^{2 + \varepsilon}\left(\sum_{N <n \leq 2N} \frac{|A(l,n)|^2}{n} +1\right),
\end{equation}
provided $N \ll T^{\frac32 + \varepsilon}$.  Supposing \eqref{eq:Sm} holds, then we obtain
\begin{equation}
H \ll T^{2 + \varepsilon} \left(\sum_{l^2 n \ll P} \frac{|A(l,n)|^2}{ln} +1\right)\ll T^{2 + \varepsilon},
\end{equation}
by the analytic properties of the Rankin-Selberg $L$-function $L(\phi \times \phi, s)$, using a standard contour integration method (Remark 12.1.8 of \cite{Goldfeld}).  Thus Theorem \ref{thm:mainthm} follows from \eqref{eq:Sm}.

\section{The large sieve}
A powerful tool for the study of moments of $L$-functions are the large sieve inequalities, and indeed we shall frequently apply them in our subsequent work.  The classical large sieve inequality for Farey fractions states
\begin{equation}
\label{eq:largesieve}
 \sum_{b \leq B} \sum_{\substack{x \shortmod{b} \\ (x,b) = 1}} \left| \sum_{N \leq m < N + M} a_m \e{xm}{b} \right|^2 \leq (B^2 + M) \sum_{N \leq m < N+ M} |a_m|^2.
\end{equation}
Gallagher \cite{Gallagher} generalized \eqref{eq:largesieve} to handle an additional integration against $m^{it}$.  The following result extends Gallagher's method to allow for more general oscillatory integrals.
\begin{mylemma}
\label{lemma:variantsieve}
 Let $f(y)$ be a continuously differentiable function on $[N, N+M]$ such that $f'$ does not vanish.  Let $X = \sup_{y \in [N, N+M]} \frac{1}{|f'(y)|}$.  Then for any complex numbers $b_m$,
\begin{equation}
\label{eq:9.2}
 \int_{-T}^{T} \sum_{b \leq B} \sum_{\substack{x \shortmod{b} \\ (x,b) = 1}} \left| \sum_{N \leq m < N + M} b_m \e{xm}{b} e(t f(m)) \right|^2 dt \ll (B^2 T + X) \sum_{N \leq m < N+ M} |b_m|^2.
\end{equation}
\end{mylemma}
The case $f(y) = \frac{1}{2 \pi} \log{y}$ is handled by \cite{Gallagher}, where $X \asymp N$, though we could not find \eqref{eq:9.2} in the literature.  For the proof we will modify a trick we learned from the paper \cite{Luo}.
\begin{proof}
By the change of variables $t \rightarrow tT$, it suffices to consider the case $T=1$.  Let $w$ be a nonnegative Schwartz function such that $w(x) \geq 1$ for $|x| \leq 1$, such that $\widehat{w}$ has compact support.  See \cite{Vaaler} for a nice survey on such functions as well as some ideas relevant in this proof. Then for any sequence of complex numbers $c_m$, we have 
\begin{equation}
\label{eq:9.3}
\int_{-1}^{1} \left| \sum_m c_m e(t f(m)) \right|^2 dt \leq \sum_{m,n} c_m \overline{c_n} \widehat{w}(f(m) - f(n)).
\end{equation}
Since $\widehat{w}$ is compactly supported, we must have $|f(m) - f(n)| \ll 1$.  By the mean-value theorem, $|f(m) - f(n)| \geq |m-n| \inf_y |f'(y)|$, so $|m-n| \ll X$.  Dissect the sum over $m$ and $n$ into boxes $I \times J$ of sidelength $\ll \min(M, X)$ so that the only relevant boxes $I \times J$ have $I$ and $J$ either equal or adjacent (``nearby'', say).  Thus the right hand side of \eqref{eq:9.3} equals
\begin{equation}
\label{eq:9.4}
 \sum_{I, J \text{ nearby}} \sum_{(m,n) \in I \times J} c_m \overline{c_n} \widehat{w}(f(m) - f(n)).
\end{equation}
Having enforced the condition that $I$ and $J$ are nearby, we then reverse the Fourier transform to express it in terms of $w$, getting that \eqref{eq:9.4} equals
\begin{equation}
\label{eq:9.5}
\intR w(t) \sum_{I, J \text{ nearby}} \sum_{(m,n) \in I \times J} c_m e(tf(m)) \overline{c_n e(tf(n))} dt.
\end{equation}
By Cauchy's inequality, \eqref{eq:9.5} is
\begin{equation}
 \ll \intR w(t) \sum_{I} \left| \sum_{m \in I} c_m e(t f(m)) \right|^2 dt.
\end{equation}
Specializing this to $c_m = \e{xm}{b}$ and summing over $x$ and $b$ appropriately gives that the left hand side of \eqref{eq:9.2} (with $T=1$) is
\begin{equation}
\label{eq:9.7}
 \ll \intR w(t) \sum_{I} \sum_{b \leq B} \sum_{\substack{x \shortmod{b} \\ (x,b) = 1}} \left| \sum_{m \in I} b_m \e{xm}{b} e(t f(m)) \right|^2 dt.
\end{equation}
By \eqref{eq:largesieve} with $a_m = b_m e(t f(m))$, we get that \eqref{eq:9.7} is
\begin{equation}
 \intR |w(t)| \sum_{I} (B^2 + \min(M, X)) \sum_{m \in I} |b_m|^2 \ll (B^2 + X) \sum_{N \leq m < N + M} |b_m|^2. \qedhere
\end{equation}
\end{proof}

\section{Spectral large sieve at special points}
In \cite{Luo}, W. Luo obtained a strong large sieve inequality applicable to the analysis of $H_l$.  For any sequence of real numbers $a_n$, let
\begin{equation}
 S(\mathcal{A}) = \sum_{t_j} w(t_j) |\rho_j(1)|^2  \left| \sum_{N < n \leq 2N} a_n \lambda_{j}(n) n^{i t_j} \right|^2,
\end{equation}
as in (15) of \cite{Luo}.  Luo's result is that
\begin{equation}
\label{eq:Luo}
 S(\mathcal{A}) \ll (T^2 + T^{3/2} N^{1/2} + N^{5/4})(NT)^{\varepsilon} \|\mathcal{A}\|^2.
\end{equation}
The term $T^{3/2} N^{1/2}$ is suggestive of the estimate \cite[Corollary 12.1]{IL}, and indeed there are some analogies between these large sieve inequalities.  Taking $N = T^{3/2 + \varepsilon}$ gives the estimate of $T^{9/4 + \varepsilon}$ for \eqref{eq:sixthmoment}.

To prove \eqref{eq:Luo}, Luo decomposes the sum $S(\mathcal{A})$ into the sum of other terms and estimates the various terms in different ways, with the primary tools being the Kuznetsov formula and the classical additive large sieve inequality.  All of these estimates but one are sufficient for our application. 
Our approach differs in that we shall obtain a new asymptotic evaluation of this exceptional term that is conducive to further analysis using special properties of the vector $a_n$.  Precisely, we shall show
\begin{mytheo}
\label{thm:largesieve}
 We have for any $1 \leq X \leq T$ and $N \gg T$,
\begin{equation}
\label{eq:SS1}
 S(\mathcal{A}) = S_1(\mathcal{A};X) + O(T^2 + \frac{NT}{X} + \frac{N^{3/2}}{T} )N^{\varepsilon} \|\mathcal{A}\|^2,
\end{equation}
where 
\begin{equation}
\label{eq:S1bound}
 S_1(\mathcal{A};X) \ll T \sum_{r < X} \frac{1}{r} \sum_{0 \neq |k| \ll rT^{\varepsilon}} \frac{1}{|k|} \int_{-T^{-\varepsilon}}^{T^{-\varepsilon}} \left| \sum_n a_n S(k,n;r) \e{un}{rT} \right|^2 du.
\end{equation}
\end{mytheo}
Remarks.
\begin{itemize}
 \item The similarity of the term $S_1(\mathcal{A};X)$ to the main term of \cite{IL}, Theorem 1.1 is striking, though there are significant differences.  In particular, no Bessel function appears in \eqref{eq:S1bound}, and to create a bilinear form we had to use the Fourier method to separate variables (whence there is some possible loss in this representation).  It is unclear if $S(\mathcal{A})$ contains a lower-order main term (for special choices of $a_n$ of course) as in \cite{IL}.

 \item 
For our application to Theorem \ref{thm:mainthm}, the error term of \eqref{eq:SS1} is sufficient.  In Section \ref{section:Voronoi} we use special properties of $a_n$ to estimate $S_1(\mathcal{A};X)$, using $GL(3)$ Voronoi summation.
\item We can recover \eqref{eq:Luo} from \eqref{eq:SS1} sketched as follows.  Use $|k|^{-1} \leq 1$ and extend the sum over $k$ to $\ll T^{\varepsilon}$ complete sums modulo $r$.  Opening the square and computing the sum over $k$ shows that \eqref{eq:S1bound} is
\begin{equation}
\ll T^{1 + \varepsilon} \sum_{r < X} \sumstar_{h \shortmod{r}} \int_{-1}^{1} \left| \sum_n a_n \e{hn}{r} \e{un}{rT}\right|^2 du, 
\end{equation}
which is bounded by $T^{1 + \varepsilon} (X^2 + XT) \sum_n |a_n|^2$ by Lemma \ref{lemma:variantsieve}.
If $N \leq T^3$ then choosing $X = \sqrt{N/T} \leq T$ and then replacing $T$ by $T + \sqrt{N}$ (noting that $S(\mathcal{A})$ is increasing in $T$) gives \eqref{eq:Luo}, while if $N > T^3$ then we take $X = T$ and replace $T$ by $N^{3/8} + T$.

\end{itemize}
\subsection{Summary of Luo's results}
In this section we briefly summarize Luo's results from \cite{Luo}, referring to this well-written paper for proofs.

The expression corresponding to $S(\mathcal{A})$ arising from the continuous spectrum is
\begin{equation}
 T(\mathcal{A}) = \frac{1}{4 \pi} \intR w(t) \left| \sum_n a_n \eta_t(n) n^{it}\right|^2 dt,
\end{equation}
where $\eta_t(n) = \sum_{ab = n} \leg{a}{b}^{it}$.  Luo shows directly \cite[Proposition 1]{Luo} that if $N \gg T$ then
\begin{equation}
\label{eq:TA}
 T(\mathcal{A}) = \frac{1}{\pi \tan{\delta}} \sum_m \sum_n a_m a_n \sigma(m,n) + O((N + T^2)N^{\varepsilon} \|\mathcal{A}\|^2),
\end{equation}
where $\|\mathcal{A}\|^2 = \sum_n |a_n|^2$, and
\begin{equation}
 \sigma(m,n) = \sum_{r=1}^{\infty} r^{-2} S(0,m;r) S(0,n;r).
\end{equation}
The term $r=1$ shows $T(\mathcal{A}) \gg T |\sum_n a_n|^2$ which can be $\gg TN \|\mathcal{A}\|^2$ for certain sequences, so it is important to separate the evaluations from the discrete and continuous spectra.  In the case where $N \ll T$ then an earlier result implicit in \cite{DI} (via the proof of their Theorem 6) furnishes the estimate
\begin{equation}
 S(\mathcal{A}) + T(\mathcal{A}) \ll T^{2 + \varepsilon} \|\mathcal{A}\|^2,
\end{equation}
so we henceforth suppose $N \gg T$.

The Kuznetsov formula expresses $S(\mathcal{A}) + T(\mathcal{A})$ as a sum of Kloosterman sums (plus a ``diagonal'' term), say
\begin{equation}
\label{eq:Kuznetsov}
 S(\mathcal{A}) + T(\mathcal{A})  = \frac{2}{\pi^2} T^2 \|\mathcal{A}\|^2 + \text{Re}(P(\mathcal{A})) + O(\|\mathcal{A}\|^2),
\end{equation}
where for a certain integral transform $f_A$ of $w$,
\begin{equation}
 P(\mathcal{A}) = \sum_{m,n} a_m a_n \sum_{c=1}^{\infty} \frac{S(m,n;c)}{c} f_{A}\leg{4 \pi \sqrt{mn}}{c}.
\end{equation}
The Hankel inversion formula gives $f = f_A + f_B$, and correspondingly
\begin{equation}
\label{eq:P=Q-R}
P(\mathcal{A}) = Q(\mathcal{A}) - R(\mathcal{A}).
\end{equation}
By \cite[(47)]{Luo},
\begin{equation}
\label{eq:RA}
 R(\mathcal{A}) \ll (N + T^2 + \frac{N^{3/2}}{T}) N^{\varepsilon} \|\mathcal{A}\|^2.
\end{equation}
Furthermore,
\begin{equation}
\label{eq:QQ00}
 Q(\mathcal{A}) = 2 \cos(\delta) Q_{00}(\mathcal{A}) + O(\frac{NT}{X} N^{\varepsilon}\|\mathcal{A}\|^2),
\end{equation}
where $\delta = (2T)^{-1}$ and $1 \leq X \leq T$ is a parameter to be chosen later.  Here
\begin{equation}
\label{eq:Q00''}
 Q_{00}(\mathcal{A}) = \sum_{m \neq n} a_m a_n |m-n| \sum_{r < X} r^{-2} \sum_{q=1}^{\infty} q^{-2} V_{-q}(m,n;r) \exp\leg{-y}{qr},
\end{equation}
with $y = 2 \pi |m-n| \sin{\delta}$ and
\begin{equation}
 V_{d}(m,n;r) = \sum_{\substack{s \shortmod{r} \\  (s(d+s),r) = 1}} \e{m\overline{s} - n(\overline{d+s})}{r},
\end{equation}
in the recent notation of \cite{IL}.  Next Luo uses the Euler-Maclaurin formula to the sum over $q$ modulo $r$  appearing in \eqref{eq:Q00''}.  The main term obtained by replacing the sum by an integral gives an expression that largely cancels the sum of Ramanujan sums appearing in \eqref{eq:TA}.  He estimates the remainder term using the large sieve and obtains \eqref{eq:Luo}.  To improve on Luo's estimate, we will analyze $Q_{00}(\mathcal{A})$ using the Fourier method with inspiration from \cite{IL}.

\subsection{Development of $Q_{00}(\mathcal{A})$}
We use Poisson summation on the sum over $q$ modulo $r$.  Actually for important technical reasons we first introduce a smooth nondecreasing weight function $\eta(q)$ satisfying $\eta(t) = 1$ for $t \geq 1$ and $\eta(t) = 0$ for $t \leq 1/2$ (which does not alter the sum of course). Hence
\begin{multline}
 \sum_{q=1}^{\infty} q^{-2} \eta(q) V_{-q}(m,n;r) \exp\leg{-y}{qr} 
\\
= \frac{1}{r} \sum_{a \shortmod{r}} V_{-a}(m,n;r) \sum_{k \in \mz} \e{ka}{r} \int_0^{\infty} t^{-2} \eta(t) \exp\leg{-y}{tr} \e{-kt}{r} dt.
\end{multline}
By a simple direct computation (or see \cite[Lemma 3.1]{IL} for a more general result),
\begin{equation}
 \sum_{a \shortmod{r}} V_{-a}(m,n;r) \e{ak}{r} = S(k,m;r)S(k,n;r).
\end{equation}
Write 
\begin{equation}
\label{eq:Q00decomp}
Q_{00}(\mathcal{A}) = Q_{000}(\mathcal{A}) + Q_{001}(\mathcal{A}),
\end{equation}
corresponding to the terms with $k=0$ and $k \neq 0$, respectively.  We have
\begin{equation}
 Q_{000}(\mathcal{A}) = \sum_{m \neq n} a_m a_n  \sum_{r < X} r^{-2}   S(0,m;r)S(0,n;r) \frac{|m-n|}{r} \int_0^{\infty} t^{-2} \eta(t) \exp\leg{-y}{tr} dt.
\end{equation}
The rest of this section is devoted to proving the following
\begin{mylemma} 
\label{lemma:Q000}
For any $1 \leq X \leq T$ and $N \gg T$, we have
\begin{equation}
\label{eq:Q000}
 2 \cos(\delta) Q_{000}(\mathcal{A}) = T(\mathcal{A}) + O\left(T^{2} + \frac{NT}{X} \right)N^{\varepsilon} \|\mathcal{A}\|^2.
\end{equation}
\end{mylemma}
By gathering the estimates \eqref{eq:TA}, \eqref{eq:Kuznetsov}, 
\eqref{eq:RA}, \eqref{eq:QQ00}, \eqref{eq:Q00decomp}, \eqref{eq:Q000}, and renaming $Q_{001}(\mathcal{A}) = L(\mathcal{A};X)$, 
we obtain 
\begin{myprop}
\label{prop:SA}
For any $1 \leq X \leq T$ and $N \gg T$, we have
\begin{equation}
 S(\mathcal{A}) = 
S_1(\mathcal{A};X)
+ O\left(T^{2} + \frac{NT}{X}+ \frac{N^{3/2}}{T}\right) N^{\varepsilon} \|\mathcal{A}\|^2,
\end{equation}
where $S_1(\mathcal{A};X) = 2 \cos(\delta) L(\mathcal{A};X)$, and
\begin{multline}
\label{eq:LAX}
 L(\mathcal{A};X) = \sum_{m \neq n} a_m a_n \sum_{r < X} r^{-2} \sum_{k \neq 0} \frac{|m-n|}{r} S(k,m;r)S(k,n;r) 
\\
\int_0^{\infty} t^{-2} \eta(t) \exp\leg{-2\pi |m-n| \sin{\delta}}{tr} \e{-kt}{r} dt.
\end{multline}
\end{myprop}

\begin{proof}[Proof of Lemma \ref{lemma:Q000}]
We approximate the weight function with simple manipulations, letting $\eta(t^{-1}) - 1 = \psi(t)$:
\begin{align}
 \int_0^{\infty} t^{-2} \eta(t)  \exp\leg{-y}{tr} dt 
&= \int_0^{\infty} t^{-2}  \exp\leg{-y}{tr} dt + \int_0^{\infty} t^{-2} (\eta(t)-1) \exp\leg{-y}{tr} dt
\\
&= \frac{r}{y} + \int_1^{\infty} \psi(t) \exp\leg{-yt}{r} dt = \frac{r}{y} + \frac{r}{y}\int_1^{2} \psi'(t) \exp\leg{-yt}{r} dt.
\end{align}
Write $Q_{000}(\mathcal{A}) = Q_{0000}(\mathcal{A}) + Q_{0001}(\mathcal{A})$ accordingly, so that
we have for the main term 
\begin{multline}
Q_{0000}(\mathcal{A}) = \frac{1}{2 \pi \sin{\delta}} \sum_{m \neq n} a_m a_n \sum_{r < X} r^{-2} S(0,m;r) S(0,n;r) 
\\
= \frac{1}{2 \pi \sin{\delta}} \sum_{r=1}^{\infty} r^{-2} \sum_{m,  n} a_m a_n  S(0,m;r) S(0,n;r) + O(\frac{NT}{X} N^{\varepsilon}\|\mathcal{A}\|^2),
\end{multline}
estimating the tail of the sum with
\begin{equation}
 \sum_m \sum_n (m,n) |a_m| |a_n| \ll N^{1 + \varepsilon} \sum_n |a_n|^2.
\end{equation}

Notice that this sum, times $2 \cos{\delta}$, gives the same sum of Ramanujan sums appearing in \eqref{eq:TA}.
Furthermore notice
\begin{equation}
Q_{0001}(\mathcal{A}) =  \int_1^{2} \frac{\psi'(t)}{2 \pi \sin{\delta}}  \sum_{r < X} r^{-2}  \sum_{m \neq n} a_m a_n S(0,m;r) S(0,n;r) \exp\left(\frac{- 2\pi |m-n| \sin{\delta} }{r} t \right) dt.
\end{equation}
In the following subsection we show $Q_{0001}(\mathcal{A}) \ll  T^{2}N^{\varepsilon} \|\mathcal{A}\|^2$.  Given this estimate, clearly Lemma \ref{lemma:Q000} follows. \end{proof}

\subsection{Estimation of $Q_{0001}(\mathcal{A})$}
\label{section:6.3}
The basic idea is to use the large sieve after separating variables.  
First note that extending the sum to include $m=n$ using $|S(0,m;r)| \leq r$ gives an acceptable error term. Divide the sum over $r$ into dyadic intervals, say $R < r \leq 2R$ and then apply the change of variables $t \rightarrow rt /(RT \sin{\delta})$ to reduce to estimating
\begin{equation}
\frac{T}{R^2} \sum_{R = 2^j \leq X} \sum_{R < r \leq 2R} \int_{1/8}^{1} \left|\sum_{m, n} a_m a_n S(0,m;r) S(0,n;r) \exp\left(\frac{- 2\pi |m-n| }{RT} t \right) \right|dt.
\end{equation}
Then separate the variables $m$ and $n$ using the Fourier integral
\begin{equation}
 \exp(-2\pi|x|) = \frac{1}{\pi} \intR e(xv) \frac{1}{1 + v^2} dv.
\end{equation}
Thus using Cauchy's inequality, the estimation of $Q_{0001}(\mathcal{A})$ reduces to the estimation of
\begin{equation}
\label{eq:Q0001}
\sum_{R =2^j \leq X} \frac{T}{R^2}  \intR \frac{1}{1 + v^2}   \sum_{R < r \leq 2R}   \left|\sum_{m}  a_m S(0,m;r) e^{\frac{imv}{RT}} \right|^2 dv. 
\end{equation}
Next open the Ramanujan sum and apply Cauchy's inequality to obtain that 
\begin{align}
 \sum_{R < r \leq 2R} \left|\sum_{m}  a_m S(0,m;r) \right|^2 &\leq 
\sum_{R < r \leq 2R} r \sumstar_{h \shortmod{r}} \left|\sum_{m} a_m  \e{hm}{r} \right|^2.
\end{align}
Inserting this into \eqref{eq:Q0001} and applying Lemma \ref{lemma:variantsieve} easily gives the desired estimate of $Q_{0001}(\mathcal{A}) \ll T^{2}N^{ \varepsilon} \|\mathcal{A}\|^2$.

\section{Separation of variables}
In order to deduce Theorem \ref{thm:largesieve} from Proposition \ref{prop:SA}, we need to develop the properties of the weight function appearing in \eqref{eq:LAX}, a task we perform presently.

The weight function $\eta$ plays a role in effectively truncating the sum over $k$ appearing in $L(\mathcal{A};X)$.  Indeed, if one replaces $\eta(t)$ by $1$, then the Fourier transform appearing in \eqref{eq:LAX} can be written in terms of the $K_1$-Bessel function (see  3.324.1 of \cite{GR}), which would truncate $k$ at a position much larger than if one retains $\eta(t)$.  Generally speaking, using $\eta$ simplifies various arguments by regularizing the weight functions at a singularity.

\subsection{Estimates for the weight function}
By a change of variables, we have
\begin{multline}
\label{eq:weightfunction}
 \frac{|m-n|}{r} \int_0^{\infty} t^{-2} \eta(t) \exp\leg{-2\pi |m-n| \sin{\delta}}{tr} \e{-kt}{r} dt 
\\
= \frac{1}{2 \pi \sin{\delta}} \int_0^{\infty} t^{-2} \eta\leg{2\pi |m-n|t \sin{\delta}   }{r} e^{-t^{-1}} \e{-2\pi k \sin{\delta} |m-n| t}{r^2} dt.
\end{multline}
Let $x = m-n$, $A = r/(2\pi \sin{\delta})$ and $B = r^2/(2\pi k \sin{\delta} )$, so that \eqref{eq:weightfunction} takes the form
$(2 \pi \sin{\delta})^{-1}W_{A,B}(x)$,
where
\begin{equation}
\label{eq:Wdef}
 W_{A,B}(x) = \int_0^{\infty} t^{-2} \eta\leg{|x|t}{A} e^{-t^{-1}} \e{-|x|t}{B} dt.
\end{equation}
We will use the Fourier method to separate the variables $m$ and $n$ in $W(x)$.  More specifically, we shall show the following estimates in Section \ref{section:fourier}.  We have
\begin{equation}
\label{eq:7.3}
\frac{1}{A} \widehat{W}_{A,B}\leg{u}{A} \ll \min\left(\frac{1}{|u|}, \frac{|B|/A}{1 + u^2} \right),
\end{equation}
that 
\begin{equation}
\label{eq:7.4}
 W_{A,B}(x) =\intR \widehat{W}_{A,B}(u) e(ux) du = \intR \frac{1}{A} \widehat{W}_{A,B}\leg{u}{A} \e{ux}{A} du.
\end{equation}
and that for any nonnegative integer $K$,
\begin{equation}
\label{eq:7.5}
 W_{A,B}(x) \ll_K \left(1 + \frac{A+|x|}{|B|} \right)^{-K}.
\end{equation}

\subsection{Deducing Theorem \ref{thm:largesieve} from Proposition \ref{prop:SA}}
Taking the estimates \eqref{eq:7.3}-\eqref{eq:7.5} temporarily for granted, we now complete the proof of Theorem \ref{thm:largesieve}.

In our case, $A/B = k/r$ so in practice we may assume $k \ll r N^{\varepsilon}$.
Thus we have
\begin{multline}
 L(\mathcal{A};X) =  \frac{1}{2\pi \sin{\delta}} \sum_{m \neq n} a_m a_n \sum_{r < X} r^{-2} \sum_{0 \neq k \ll r N^{\varepsilon}} S(k,m;r)S(k,n;r) 
\\
 \intR A^{-1} \widehat{W}_{A,B}\leg{u}{A} \e{-2\pi \sin{\delta} (m-n)u}{r} du + O(N^{-2008} \|\mathcal{A}\|^2).
\end{multline}
Using the Weil bound shows that we can extend the summation to $m=n$ with an error of size $\ll XT^{1+\varepsilon} \|\mathcal{A}\|^2$.  Thus we have
\begin{multline}
\label{eq:7.8}
 L(\mathcal{A};X) \ll  T  \sum_{r < X} r^{-2} \sum_{0 \neq k \ll r T^{\varepsilon}} 
\intR A^{-1} \left|\widehat{W}\leg{u}{A}\right|  
\left| \sum_{ n}  a_n S(k,n;r) \e{2\pi \sin{\delta}u n}{r} \right|^2 du 
\\
+ O(XT N^{\varepsilon} \|\mathcal{A}\|^2).
\end{multline}
Let $L_1(\mathcal{A};X)$ be the first term of \eqref{eq:7.8}.  Changing variables and using \eqref{eq:7.3} gives
\begin{equation}
 L_1(\mathcal{A};X) \ll T \sum_{r < X} r^{-2} \sum_{0 \neq k \ll r N^{\varepsilon}}  \intR \min\left(\frac{1}{|u|}, \frac{r/|k|}{1 + u^2} \right)\left| \sum_{ n}  a_n S(k,n;r) \e{u n}{r T} \right|^2 du.
\end{equation}
For $|u| \leq T^{-\varepsilon}$ we bound the minimum by the latter term, which gives the right hand side of \eqref{eq:S1bound}, so it suffices to show that the contribution from the rest of the region of integration is accounted by the error term of Theorem \ref{thm:largesieve}. 

The integration for $|u| \geq XT^{1 + \varepsilon}$ picks out the diagonal (with a negligible error), giving
\begin{equation}
 \ll T \sum_{r < X} r^{-1} \sum_{0 \neq k \ll r N^{\varepsilon}} \frac{1}{|k|} \frac{1}{XT} \sum_n |a_n S(k,n;r)|^2 \ll N^{\varepsilon} \sum_n |a_n|^2,
\end{equation}
using Weil's bound.  For $T^{-\varepsilon} \leq |u| \leq XT^{1 + \varepsilon}$, dissect the region of integration into $\ll \log{N}$ dyadic intervals $U \leq |u| \leq 2U$, getting for such a dyadic interval
\begin{equation}
\label{eq:8.10} 
T \sum_{r < X} r^{-2} \sum_{0 \neq k \ll r N^{\varepsilon}} \frac{1}{U} \int_{U \leq |u| \leq 2U} \left| \sum_n a_n S(k,n;r) \e{un}{rT} \right|^2 du.
\end{equation}
Extending the sum over $k$ to $\ll N^{\varepsilon}$ complete sums modulo $k$, opening the square, opening the Kloosterman sums, and executing the summation over $k$ gives that \eqref{eq:8.10} is
\begin{equation}
 \ll T \sum_{r < X} r^{-1} \sumstar_{h \shortmod{r}} \frac{1}{U} \int_{U \leq |u| \leq 2U} \left| \sum_n a_n \e{hn}{r} \e{un}{rT} \right|^2 du,
\end{equation}
which by Lemma \ref{lemma:variantsieve} is
\begin{equation}
 \ll (XT + U^{-1}T^2) N^{\varepsilon} \sum_n |a_n|^2.
\end{equation}
This proves Theorem \ref{thm:largesieve}.

\subsection{Proofs of the estimates on $W_{A,B}(x)$}
\label{section:fourier}
\begin{proof}[Proof of \eqref{eq:7.3}-\eqref{eq:7.5}]
 It is immediate from the definition \eqref{eq:Wdef} that $W(x)$ is continous on $\mr$.  An integration by parts argument will show presently that $W(x)$ is smooth except at the point $x=0$ (due to the absolute value) with rapid decay for $x$ large.  Precisely, we have
\begin{equation}
\label{eq:7.9}
 W_{A,B}(x) = \leg{B}{2 \pi i |x|}^K \int_0^{\infty} \frac{\partial^K}{\partial t^K}\left[t^{-2} \eta\leg{|x|t}{A} e^{-t^{-1}}  \right] \e{-|x|t}{B} dt.
\end{equation}
An exercise with the generalized product rule for derivatives (Leibniz's formula) shows that the term $\frac{\partial^K}{\partial t^K} [ \dots ]$ can be expressed as a sum of functions of the form
\begin{equation}
 t^{-2-K} \nu\leg{|x|t}{A} e^{-t^{-1}} P(t^{-1}),
\end{equation}
where $\nu(y) = y^{j} \eta^{(j)}(y)$ for some $j \leq K$, and $P(y)$ is a polynomial.  Using that $\nu(y) = 0$ for $y \leq \half$, we estimate the integral \eqref{eq:7.9} with absolute values to obtain
\begin{equation}
 W_{A,B}(x) \ll \leg{B}{|x|}^K \left(1 + \frac{A}{|x|}\right)^{-1 - K} = \frac{|x|}{A+ |x|} \leg{|B|}{A + |x|}^K,
\end{equation}
which implies 
\eqref{eq:7.5} by taking $K$ large only if $A + |x| \geq |B|$, and by taking $K=0$ otherwise. 

It is clear that $W(x)$ is absolutely integrable, and \eqref{eq:7.3} will show that its Fourier transform is also absolutely integrable, and hence by continuity $W$ can be recovered via \eqref{eq:7.4}.

Now we compute the Fourier transform of $W$ by somewhat direct calculation.  Since $W$ is even, we have
\begin{equation}
 \widehat{W}_{A,B}(u) = \int_0^{\infty} \left[ e(ux) + e(-ux)\right] W_{A,B}(x) dx.
\end{equation}
Write $\widehat{W}(u) = \widehat{W}_{+}(u) + \widehat{W}_{-}(u)$ correspondingly.  Changing variables by $t \rightarrow \frac{t}{x}$ in \eqref{eq:Wdef} gives
\begin{equation}
\label{eq:7.13}
 \widehat{W}_{+}(u) = \int_0^{\infty} e(ux) \int_0^{\infty} t^{-1} \eta\leg{t}{A} \frac{x}{t} \exp\left(-\frac{x}{t}\right) \e{-t}{B} dt dx.
\end{equation}
Letting $F(y) = y e^{-y}$, we get that this takes the form
\begin{equation}
  \widehat{W}_{+}(u) = \int_0^{\infty} e(ux) \int_0^{\infty} t^{-1} \eta\leg{t}{A} F\leg{x}{t} \e{-t}{B} dt dx.
\end{equation}
Our next goal is to reverse the order of integrations and directly compute the $x$-integral.  Unfortunately, the double integral does not converge absolutely so it takes some more involved arguments to justify this interchange (repeated integration by parts).  Let
\begin{equation}
 f(x,t) = t^{-1} \eta\leg{t}{A} F\leg{x}{t},
\end{equation}
and note that its $j$-th partial derivative with respect to $t$ 
is a sum of functions of the form
\begin{equation}
 t^{-1-j} \nu\leg{t}{A} G\leg{x}{t} P\leg{1}{t},
\end{equation}
where $\nu(y) = y^{i} \eta^{(i)}(y)$ for some $i \leq j$, $G(y)$ is a polynomial times $e^{-y}$, and $P(y)$ is a polynomial.  In particular, 
\begin{equation}
 \frac{\partial^3}{\partial t^3} f(x,t) \ll \frac{1}{(1 + t^2)(1 + x^2)}.
\end{equation}
Then after three times integration by parts we can reverse the orders of integration, getting
\begin{equation}
 \widehat{W}_{+}(u) = \int_0^{\infty} \frac{\e{-t}{B}}{(2 \pi i/B)^3}  \int_0^{\infty} \frac{\partial^3}{\partial t^3} f(x,t) e(ux) dx dt.
\end{equation}
Since the partial derivatives are continuous and have rapid decay with respect to $x$, it is easy to take the differentiation outside the integral sign, whence
\begin{equation}
 \widehat{W}_{+}(u) = \int_0^{\infty} \frac{\e{-t}{B}}{(2 \pi i/B)^3} \frac{\partial^3}{\partial t^3} \left[t^{-1} \eta\leg{t}{A}  \int_0^{\infty} \frac{x}{t} e^{-\frac{x}{t}} e(ux) dx \right] dt.
\end{equation}
Note that
\begin{equation}
 \int_0^{\infty} \frac{x}{t} e^{-\frac{x}{t}} e(ux) dx = \frac{t}{(1-2 \pi i ut)^2}.
\end{equation}
With this representation it is easy to perform integration by parts backwards to finally get
\begin{equation}
 \widehat{W}_{+}(u) = \int_0^{\infty} \eta\leg{t}{A} \e{-t}{B} \frac{1}{(1 - 2 \pi i ut)^2} dt,
\end{equation}
which is what we would obtain directly from \eqref{eq:7.13} by interchanging the integrals.  Thus we obtain after a change of variables
\begin{equation}
\label{eq:7.23}
\frac{1}{A} \widehat{W}_{A,B}\leg{u}{A} = 2 \int_0^{\infty} \eta(t) \e{-At}{B} \frac{1 - (2\pi ut)^2}{(1 + (2\pi ut)^2)^2}dt.
\end{equation}
Estimating the integral by absolute values shows that \eqref{eq:7.23} is $\ll |u|^{-1}$, while a single integration by parts gives the other estimate of \eqref{eq:7.3}.
\end{proof}

\section{Voronoi summation}
\label{section:Voronoi}
We need to estimate $S_1(\mathcal{A};X)$ in the case where
\begin{equation}
 a_n = \frac{A(l,n)}{\sqrt{n}} w_2\leg{n}{N} = \frac{1}{\sqrt{N}} A(l,n) w_3\leg{n}{N},
\end{equation}
where $w_3(x) = x^{-\half} w_2(x)$ is also smooth and compactly supported on $[N, 2N]$ (actually the previous estimates of Luo required $a_n$ real so we need to consider the real and imaginary parts of $A(l,n)$ separately). For this choice of $\mathcal{A}$, let $S_1(X) = S_1(\mathcal{A};X)$.  Our plan is to apply the $GL(3)$ Voronoi summation formula proved by \cite{MS}.
To this end, let
\begin{equation}
 C=C(k,l,r,u,T) = \frac{1}{\sqrt{N}} \sum_{n} A(l,n) S(k,n;r) w_3\leg{n}{N} \e{un}{r T},
\end{equation}
so that
\begin{equation}
 S_1(X) \ll T \sum_{r < X} r^{-1} \sum_{0 \neq k \ll rN^{\varepsilon}} |k|^{-1} \int_{-T^{-\varepsilon}}^{T^{-\varepsilon}} |C(k,l,r,u,T)|^2 du.
\end{equation}

Note that since the Kloosterman sums are real we could easily reduce the sum of real (or imaginary) parts of $A(m,n)$ to the sum of $A(m,n)$.

It turns out that for the relevant ranges of the variables that $C$ is negligibly small.
\begin{mylemma}
\label{lemma:L1}
 Let $X = (N/l)^{1/3} T^{-\varepsilon}$.  Then
\begin{equation}
\label{eq:L1}
 S_1(X) \ll T^{-2009}.
\end{equation}
\end{mylemma}
For this choice of $X$, notice that the error term of Theorem \ref{thm:largesieve}  is
\begin{equation}
\label{eq:Sfinal}
 \ll T^{\varepsilon}(N + T^2 + N^{2/3} T l^{1/3} + \frac{N^{3/2}}{T}) \|\mathcal{A}\|^2,
\end{equation}
and that the error term of \eqref{eq:7.8} is absorbed by \eqref{eq:Sfinal}.
Since $N = P/l^2 \ll T^{3/2 + \varepsilon}/l^2$, \eqref{eq:Sfinal} is $\ll T^{2 + \varepsilon} \sum_{N < n \leq 2N} \frac{|A(l,n)|^2}{n}$.  Thus combining \eqref{eq:L1} and \eqref{eq:Sfinal} gives \eqref{eq:Sm}.

\begin{proof}{Proof of Lemma \ref{lemma:L1}}.
Opening the Kloosterman sum, we get
\begin{equation}
 C = \frac{1}{\sqrt{N}} \sumstar_{h \shortmod{r}} \e{hk}{r} \sum_{n} \e{\overline{h}n}{r} A(l,n)  w_3\leg{n}{N} \e{un}{r T}.
\end{equation}
An application of the Voronoi summation formula gives
\begin{equation}
 C = \frac{r}{\sqrt{N}}  \sumstar_{h \shortmod{r}} \e{hk}{r} \sum_{\epsilon = \pm 1} \sum_{m_1 | rl} \sum_{m_2 >0} \frac{A(m_2, m_1)}{m_1 m_2} S\left(lh,\epsilon m_2;\frac{lr}{m_1}\right) \Psi^{\epsilon}\leg{m_2 m_1^2}{r^3 l}.
\end{equation}
for certain functions $\Psi^{\epsilon}(x)$.  
We will now determine the size of $\Psi^{\epsilon}$, showing that it is negligibly small for the variables in the relevant ranges.

Here each $\Psi^{\epsilon}(x)$ is expressed as a linear combination of two other functions $\Psi_0$ and $\Psi_1$.  As X. Li notes \cite{Li}, $\Psi_1$ has similar asymptotic behavior to $\Psi_0$, so we only consider $\Psi_0$.  Then we utilize her Lemma 2.1 (a generalization of Lemma 3 of Ivi\'{c} \cite{Ivic}) which gives the asymptotic behavior of $\Psi_0(x)$.  Precisely, we have
\begin{equation}
 \Psi_0(x) = x \int_0^{\infty} w_3\leg{y}{N} \e{uy}{r T} \sum_{j=1}^{K} \frac{c_j \cos(6 \pi (xy)^{1/3}) + d_j \sin(6 \pi (xy)^{1/3})}{(xy)^{j/3}} dy + O((xN)^{\frac{-K + 2}{3}}),
\end{equation}
provided $xN \gg 1$, where $c_j$ and $d_j$ are certain absolute constants.  Since $r < X$, then $xN \geq T^{\varepsilon}$ for the choice $X = (N/l)^{1/3} T^{-\varepsilon}$.
We need to analyze the integral 
\begin{multline}
 x \int_0^{\infty} (xy)^{-1/3} w_3\leg{y}{N} \e{uy}{R T} e\left(3 (xy)^{1/3}\right) dy 
\\
= (N x)^{2/3} \int_0^{\infty}  \frac{w_3(y)}{y^{1/3}}  e\left(\frac{uyN}{r T} + 3 (xyN)^{1/3}\right) dy.
\end{multline}
The integral is negligibly small unless there is a stationary point near the support of $w_3$, on integration by parts.  If $f(y) = \frac{uyN}{rT} + 3 (xyN)^{1/3}$ then
\begin{equation}
 f'(y) = \frac{uN}{r T} + \frac{(Nx)^{1/3}}{y^{2/3}}. 
\end{equation}
The stationary point $y_0$ is at
\begin{equation}
 y_0 =  \frac{x^{1/2} (rT)^{3/2}}{N |u|^{3/2}},
\end{equation}
The condition that $y_0 \asymp 1$ 
means
\begin{equation}
x \asymp \frac{N^2 |u|^3}{r^3 T^3}  = \frac{P^2 |u|^3}{r^3 T^3 l^4}  \ll \frac{T^{-\varepsilon}}{r^3 l^4}.
\end{equation}
Hence $m_1^2 m_2 \ll T^{-\varepsilon}/l^3$, and so the sum is empty.  To be more precise, the magnitude of the weight functions $\Psi^{\epsilon}(x)$ is $\ll T^{-A}$ for any $A > 0$, whence \eqref{eq:L1} follows.
\end{proof}

\section{Proof of Theorem \ref{thm:sixthmoment}}
\label{section:sixth}
Here we briefly sketch what changes are necessary to prove Theorem \ref{thm:sixthmoment}.  The high-level explanation is that we should replace the $GL(3)$ Maass form by an appropriate $GL(3)$ Eisenstein series and use a Voronoi summation formula for the coefficients of the Eisenstein series.  Unfortunately, such a general Voronoi formula does not appear to exist in the literature, though work of Ivi\'c \cite{Ivic} is sufficient here.

A replacement for \eqref{eq:Rankin} is the following Dirichlet series
\begin{equation}
 L_j(s)^3 = \sum_{a,b \geq 1}^{\infty} \frac{\mu(a) d_3(b)}{(ab)^{2s}} \sum_{n=1}^{\infty} \frac{d_3(n) \lambda_j(an)}{(an)^s}.
\end{equation}
A proof of this is an exercise with the Hecke relations.
Here $a$ and $b$ play a role analogous to $l$ in the proof of Theorem \ref{thm:mainthm}.  
An easy modification of Section \ref{section:cleaning} reduces the proof of Theorem \ref{thm:sixthmoment} to showing
\begin{equation}
\label{eq:10.2}
 \sum_j w(t_j) |\rho_j(1)|^2 \left| \sum_{\substack{N < n \leq 2N}} a_n \lambda_j(n) n^{i t_j} \right|^2 \ll T^{2 + \varepsilon},
\end{equation}
where $N = P/(a^2 b^2)$, $P \ll T^{\frac32 + \varepsilon}$, and $a_n = n^{-\half} w_2(n/N) d_3(n/a)$ if $a|n$, and $a_n = 0$ otherwise.
Theorem \ref{thm:largesieve} applies to \eqref{eq:10.2}, whence it suffices to estimate
\begin{equation}
\label{eq:10.4a}
 \sum_{r < X} \frac{T}{r} \sum_{0 \neq |k| \ll r T^{\varepsilon}} \frac{1}{|k|} \int_{-T^{-\varepsilon}}^{T^{-\varepsilon}} 
\left|\sumstar_{h \shortmod{r}} \e{hk}{r} N^{-\half} \sum_n w_3\leg{an}{N} d_3(n)  \e{\overline{h} a'n}{r'} \e{ua'n}{r'T} \right|^2 du,
\end{equation}
where we set $a' = a/(a,r)$, and $r' = r/(a,r)$.
As in Section 9, we apply the analog of the Voronoi summation formula for $d_3(n)$ as encapsulated in Theorem 2 of \cite{Ivic}.  This gives
\begin{equation}
\sum_n w_3\leg{an}{N} d_3(n)  \e{\overline{h} a'n}{r'} \e{ua'n}{r'T} = M.T. + \frac{1}{r'^3} \sum_n A_3'\left(n,\frac{\overline{h} a' }{r'}\right) \Psi\leg{n}{r'^3} + \dots,
\end{equation}
where $M.T.$ denotes a certain explicit main term, $A_3$ is a certain complete exponential sum, 
and the dots indicate three other expressions of a similar shape with slightly different weight functions and possibly $h$ replaced by $-h$.  Here the weight function $\Psi(x)$ satisfies an asymptotic expansion with leading order term given by
\begin{equation}
 \Psi\leg{n}{r'^3} \sim \int_0^{\infty} w_4\leg{ay}{N}\e{uay}{rT}  \e{3 (yn)^{1/3}}{r'}  \leg{ny}{r'^3}^{-1/3} dy.
\end{equation}
Integration by parts shows that $\Psi(n/r'^3)$ is negligible, as in Section \ref{section:Voronoi}.

Now we analyze the contribution of the main term, given as $\text{Res}_{s=1} E_3\left(s,\frac{a' \overline{h}}{r'} \right) F(s)$, where $E_3\left(s,\frac{c}{d}\right)$ is a variant on the Estermann function with $d(n)$ replaced by $d_3(n)$, and
\begin{equation}
 F(s) = \int_0^{\infty} w_3\leg{ax}{N} \e{uax}{rT} x^{s-1} dx.
\end{equation}
Here $F(s)$ is holomorphic at $s=1$, and $F(1) = \frac{N}{a} \widehat{w_3}\leg{uN}{rT}$.  Furthermore, Ivi\'{c} shows that $\text{Res}_{s=1} E_3(s,c/d) \ll d^{-1 + \varepsilon}$, and that the residue does not depend on $h$ (so the sum on $h$ is a Ramanujan sum).  With this information, we see that 
this main term gives 
\begin{equation}
\label{eq:10.4}
 \ll T^{1 + \varepsilon} \sum_{r < X} r^{-1} \intR \sum_{0 \neq k \ll r T^{\varepsilon}} \frac{1}{|k|} \left|\frac{\sqrt{N}}{a} \widehat{w_4}\leg{uN}{rT} \frac{1}{r'} S(0,k;r) \right|^2 du.
\end{equation}
Using only that $\intR a^{-2} |\widehat{w_3}\leg{uN}{rT}|^2 du \ll \frac{rT}{a^2N}$, $\sum_{1 \leq k \leq r} |k|^{-1}|S(0,k;r)|^2 \ll r^{1 + \varepsilon}$, and $r' \geq r/a$, gives that \eqref{eq:10.4} is $\ll T^{2 + \varepsilon}$, as desired.

\end{document}